\documentclass[12pt]{amsart}

\usepackage[dvips]{color}
\usepackage{amsmath}
\usepackage{amsxtra}
\usepackage{amscd}
\usepackage{amsthm}
\usepackage{amsfonts}
\usepackage{amssymb}
\usepackage[mathscr]{eucal}
\usepackage{epsfig}
\usepackage{graphics}
\usepackage{accents}
\usepackage{physics}
\usepackage{hyperref}
\usepackage{mathtools}
\usepackage{enumitem}
\usepackage{tikz}
\usepackage{stmaryrd}
\usepackage{verbatim}
\usepackage[normalem]{ulem}

\numberwithin{equation}{section}
\newtheorem{thm}{Theorem}[section]
\newtheorem{prop}[thm]{Proposition}
\newtheorem{lem}[thm]{Lemma}
\newtheorem{rem}[thm]{Remark}
\newtheorem{cor}[thm]{Corollary}

\newtheorem{dfn}[thm]{Definition}

\newcommand{\C}{{\mathbb C}}

\newcommand{\Z}{{\mathbb Z}}
\newcommand{\Zz}{{\mathbb Z}^\times}
\newcommand{\Zg}{{\mathbb Z}_{>0}}
\newcommand{\Zge}{{\mathbb Z}_{\geq 0}}
\newcommand{\A}{{\mathbb A}}

\newcommand{\B}{{\mathcal B}}

\newcommand{\cH}{\mathcal H}

\newcommand{\cP}{\mathcal{P}}

\newcommand{\cU}{\mathcal{U}}

\newcommand{\ol}{\overline}

\newcommand{\bs}{\boldsymbol}
\newcommand{\n}{{\mathfrak{n}}}

\newcommand{\gl}{\mathfrak{gl}}

\newcommand{\fpsl}{\mathfrak{psl}}
\renewcommand{\sl}{\mathfrak{sl}}
\newcommand{\osp}{\mathfrak{osp}}

\newcommand{\h}{\mathfrak{h}}

\newcommand{\hh}{\widehat{\mathfrak{h}}}

\newcommand{\Uqh}{U_q({\mathfrak{h}})}

\newcommand{\Ugh}{U(\widehat{\mathfrak{g}})}

\newcommand{\Uqgh}{U_q(\widehat{\mathfrak{g}})}
\newcommand{\Uqhh}{U_q(\widehat{\mathfrak{h}})}

\newcommand{\slhmn}{ \widehat{\sl}(m|n)}

\newcommand{\D}{\Delta}
\newcommand{\hD}{\hat\Delta}

\newcommand{\s}{\mathbf{s}}

\newcommand{\Sym}{\mathrm{Sym}}

\newcommand{\al}{\alpha}

\newcommand{\qK}[3]{{\left[
\begin{matrix}{\displaystyle #1\, ; \  #2}\\
{\ \,\displaystyle #3\,}\end{matrix}
\right]
}}

\newcommand{\cN}{\mathcal{N}}

\newcommand{\gh}{\widehat{\mathfrak{g}}}
\newcommand{\g}{\mathfrak{g}}

\DeclareMathOperator{\Ind}{Ind}

\newcommand{\lb}[1]{\llbracket #1 \rrbracket}
\newcommand{\htt}[1]{\mathrm{ht}( #1 )}

\begin{document}
\begin{title}[Diagonal modules]{Representations of affine  Lie superalgebras  and their quantization in type $A$ }
\end{title}
\author[L. Bezerra, L. Calixto, V. Futorny and I.Kashuba]{Luan Bezerra, Lucas Calixto, Vyacheslav Futorny  and Iryna Kashuba}

\address[bezerra.luan@gmail.com ]{\textsc{Luan Pereira Bezerra}: Institute of Mathematics and Statistics, University of S\~ao Paulo, S\~ao Paulo, Brazil}
\address[lhcalixto@ufmg.br]{\textsc{Lucas Calixto}: Department of Mathematics, Federal University of Minas Gerais, Belo Horizonte, Brazil}
\address[futorny@ime.usp.br]{\textsc{Vyacheslav Futorny}: Institute of Mathematics and Statistics, University of S\~ao Paulo, S\~ao Paulo, Brazil,\\  and  International Center for Mathematics, SUSTech, Shenzhen, China}
\address[kashuba@ime.usp.br]{\textsc{Iryna Kashuba}: Institute of Mathematics and Statistics, University of S\~ao Paulo, S\~ao Paulo, Brazil}
\maketitle

\begin{abstract}
We construct a new family of irreducible modules over any basic classical affine Kac-Moody Lie superalgebra which are induced from modules over the Heisenberg subalgebra. We also obtain irreducible deformations of these modules for the quantum affine superalgebra $U_q\slhmn$.  
\end{abstract}

\section{Introduction}
Affine Lie algebras and superalgebras are among the most important algebraic structures with ample applications in modern sciences. While representations of affine Lie algebras were extensively studied and essentially understood, the
representation theory of affine Lie superalgebras is largely undeveloped.
The paper is a part of our program of studying the induced representations for affine Lie superalgebras 
and their quantizations. This study was initiated in \cite{CF} where the theory of Verma type modules for 
affine Lie superalgebras was developed. In the current paper, we consider parabolic induction in the case when the Levi subalgebra is an abelian extension of the Heisenberg subalgebra over the field of fractions $\C(q)$ and the inducing modules are 
diagonal modules. These modules have infinite dimensional weight subspaces and nonzero central charge. They also admit a natural eigenbasis for a certain infinite family of operators. 
Our first result establishes the irreducibility criterion of generalized imaginary Verma modules $M(\lambda,V)=\Ind({\cP},{\Ugh};V)$ for affine Lie superalgebra $\gh$
 associated with a parabolic subalgebra $\cP$, whose Levi factor is a sum of the  Heisenberg subalgebra and a Cartan subalgebra, an irreducible diagonal module $V$ over the  Heisenberg subalgebra and a weight $\lambda$.

\begin{thm}\label{thm-1}
For any irreducible diagonal module $V$ over the  Heisenberg subalgebra with nonzero central charge $a$ and any weight $\lambda$ such that $\lambda(c)=a$, the generalized imaginary Verma module $M(\lambda,V)$ is irreducible.
\end{thm}

This theorem is an analog of the corresponding  result in \cite{BBFK} in the case of affine Lie algebras (see also \cite{EF}, \cite{FK1}, \cite{FK2}, \cite{FKS}). Furthermore, differently from the other cases in the literature, we point out that we work with a specific order on the monomials of $U(\gh)$. Such order is crucial in the study of quantum deformations of generalized imaginary Verma modules. Notice also that the category of diagonal modules over the  Heisenberg subalgebra with a fixed nonzero central charge is equivalent to the category of weight modules for some Weyl algebra of infinite rank, whose irreducible objects were classified in \cite{FGM14}. Thus Theorem \ref{thm-1} gives a way to construct a large family of irreducible modules for affine Lie superalgebras based on these Weyl algebra representations. Namely, it defines a family of functors (parametrized by $\lambda$)  from the category of weight modules over certain Weyl algebra of infinite rank to the category of modules over an  affine Lie superalgebra each of which preserves irreducibility as long as $\lambda(c)\neq 0$.

For a distinguished parabolic subalgebra $\cP_{\Sigma}$ of $\gh$, the corresponding induced $\gh$-module $M_{\Sigma}(\lambda, V)$ is isomorphic to a Kac module 
 $K(M_0(\lambda,V))$, where $\cP_0$ is the intersection of $\cP_{\Sigma}$  with the even part $\gh_0$ of $\gh$ and
  $M_0(\lambda,V)=\Ind({\cP_0},U(\gh_0);V)$ (see Theorem \ref{thm-Kac}). Hence, in particular,  Theorem \ref{thm-1} gives the irreducibility of the Kac module $K(M_0(\lambda,V))$.

Next, we pass to quantum affine superalgebras of type $A$. The restriction of the type is due to the fact that the PBW bases are only known for quantum algebras in type $A$ \cite{Tsy21}, which we heavily depend on.
We introduce the quantum analogs $M_q(\lambda, V_q):=\Ind({\cP_q}, {\Uqgh}; V_q)$
 of generalized imaginary Verma modules, where $V_q$ is an irreducible module over the quantum Heisenberg $\cH_q$.  
 We show that these modules are true quantum deformations of the corresponding classical modules using  the technique of $\mathbb A$-forms \cite{Lu}.  In the non-super case, this approach has been successfully used in \cite{FGM}, where it was based on Beck's basis, and in \cite{FHW}. 
 
 Our second result is the following theorem which gives the irreducibility criterion for the quantum generalized imaginary Verma modules induced from generic $V_q$ (cf. Corollary \ref{cor-irred}).
 
 \begin{thm}\label{thm-2} Let $V_q$ be limit faithful (see Definition \ref{def-limit-faithful}), and let $\lambda$ be an integral weight with $\lambda(c)\neq 0$. Then  the quantum generalized imaginary Verma $U_q\slhmn$-module $M_q(\lambda, V_q)$ is irreducible if and only if $V_q$ is irreducible.
\end{thm}

Since irreducible modules over the quantum Heisenberg subalgebra correspond to  irreducible modules over the classical   Heisenberg subalgebra (see Remark \ref{rem-Heis}), Theorem \ref{thm-2} allows us to construct a family of irreducible modules for quantum affine Lie superalgebras of type $A$ from weight irreducible modules over certain  Weyl algebra of infinite rank. 

The developed technique can be extended to more general parabolic subalgebras containing the Heisenberg subalgebra. We are going to address this in a forthcoming paper.

\section{Preliminaries}

\subsection*{Notation:} Throughout the text $q$ will denote an indeterminate and $\C(q)$ is the field of fractions of the polynomial ring $\C[q]$. We will use the subscript $q$ to indicate that a vector space is defined over $\C(q)$. The elements of an algebra over $\C(q)$ (respectively, $\C$) will be represented by upper case (respectively, lower case) letters. If $X, Y$ are sets, we denote by $X^Y$ the set of all functions $f:Y\to X$. 

\subsection{Affine Kac-Moody Lie superalgebras}
Let $\g=\g_0\oplus \g_1$ denote a basic classical Lie superalgebra of type $\gl(m,n)$, $\sl(m,n)$, $\fpsl(n,n)$, $\osp(m,n)$, $F(4)$, $G(2)$ or $D(2,1; a)$ for $a\neq 0,-1$. Let $\h\subset \g_0$ be a fixed Cartan subalgebra of $\g$ and consider the root space decomposition of $\g$
    \[
\g = \h\oplus \bigoplus_{\alpha\in \D}\g_\alpha,
    \]
where $\g_\alpha = \{x\in \g\mid [h,x]=\alpha(h)x\ \forall h\in \h\}$ and $\D=\{\alpha\in \h^*\setminus\{0\}\mid \g_\alpha\neq 0\}$. The set $\D$ is the root system of $\g$, and $\g_\alpha$ is the root space associated to the root $\alpha$. Recall that every root is either purely even or purely odd, meaning that, for every $\alpha\in \D$, we have $\g_\alpha\subset \g_0$ or $\g_\alpha\subset \g_1$, respectively. In particular, $\D=\D_0\cup\D_1$, where $\D_i=\{\alpha\in \D\mid \g_\alpha\subset \g_i\}$ for $i\in \Z_2$. It is known that $\g$ is endowed with an even supersymmetric invariant bilinear form $(\cdot |\cdot)$, and that such a form is nondegenerate if $\g $ is not isomorphic to $\sl(n,n)$.

Consider a set of simple roots $\Sigma = \{\alpha_i\mid i\in I\}\subset \D$ and the triangular decomposition $\D = \D^-\cup \D^+$ corresponding to $\Sigma$. Let $x^\pm_{\alpha}\in \g_{\pm \alpha}$ and $h_\alpha$ for $\alpha\in \D^+$ denote the Chevalley type generators of $\g$ associated to $\Sigma$ (in particular, $[x_\alpha^+,x_\alpha^-]=h_\alpha$ for all $\alpha\in \D^+$, and, if $\alpha$ is even, then $x_\alpha^\pm, h_\alpha$ is an $\sl_2$-triple). In what follows we set $x_i^\pm:=x^\pm_{\alpha_i}$ and $h_i:=h_{\alpha_i}$ for all $i\in I$. The triangular decomposition $\D = \D^-\cup\D^+$ induces a triangular decomposition $\g = \n^-\oplus \h\oplus \n^+$, where $\n^\pm$ is the subalgebra of $\g$ generated by $x_i^\pm$ for $i\in I$. We point out that if $\g$ is isomorphic to $\fpsl(n,n)$, then the set $\{h_i\mid i\in I\}$ is not linearly independent; in fact, in this case any $h_i$ can be expressed as a linear combination of the remainder elements. Finally, let $A = (A_{i,j})$ be the Cartan matrix associated to $\Sigma$.

Let $\gh=\g\otimes \C[t,t^{-1}]\oplus \C c\oplus \C d$ be the affine Kac-Moody Lie superalgebra corresponding to $\g$. An element $z\otimes t^m\in \g\otimes \C[t,t^{-1}]$ will be denoted by $z_m$ (in particular, $h_{i,m}=h_i\otimes t^m$, $x_{i,m}^\pm=x_i^\pm\otimes t^m$ and $x_{\alpha,m}^\pm=x_\alpha^\pm\otimes t^m$). Moreover, for all $x,y,\in \g$ and $m,n\in\Z$ we have
    \[
[x_m,y_n] = [x,y]_{m+n} + m\delta_{m,-n}(x|y)c,\quad [c,\gh]=0,\quad [d,x_m]=mx_m.
    \]
The Cartan subalgebra of $\gh$ is $\hh =\h\oplus \C c\oplus \C d$, and the root space decomposition of $\gh$ is given by
    \[
\gh = \hh\oplus \bigoplus_{\alpha\in \hD} \gh_\alpha
    \]
where $\hD = \{\alpha+n\delta\mid \alpha\in \D, n\in \Z\}\subset \hh^*$ with $\delta\in \hh^*$ being defined so that $\delta(\h+ \C c)=0$ and $\delta(d)=1$. The sets $\hD^{re} = \hD\setminus \{n\delta\mid n\in \Z\}$ and $\hD^{im}=\{n\delta\mid n\in \Z\}$ are called the sets of real and imaginary roots of $\gh$, respectively. It is easy to see that, for all $\alpha\in \D^+$ and $n\in \Z\setminus \{0\}$, we have $\gh_{\pm\alpha+n\delta}=\C x_{\alpha,n}^\pm$ and $\gh_{n\delta}=\h\otimes t^n$. Recall also that the bilinear form $(\cdot | \cdot)$ of $\g$ can be extended to an even supersymmetric invariant  bilinear form on $\gh$, which is nondegenerate if $\g$ is not isomorphic to $\sl(n,n)$; namely we have
    \[
(x_m|y_n)=\delta_{m,-n}(x|y),\quad (\g\otimes\C[t,t^{-1}]\,|\, \C c \oplus \C d)=0,\quad (c|c)=(d|d)=0,\quad (c|d)=1.
    \]
    
If $\g$ is not isomorphic to $\sl(n,n)$, then the subalgebra
    \[
\cH' = \C c\oplus \bigoplus_{n\in \Z\setminus\{0\}} \gh_{n\delta}\subset \gh
    \]
is isomorphic to an infinite-dimensional $\Z$-graded Heisenberg Lie algebra. If $\g$ is isomorphic to $\sl(n,n)$, then $\h$ contains the identity matrix $z := I_{2n,2n}$ of $\g$, and the subalgebra $\cH$ is isomorphic to a direct sum of an infinite-dimensional $\Z$-graded Heisenberg Lie algebra and a commutative Lie algebra generated by the vectors $z_r$, $r\in \Zz$.

In what follows we will consider the algebra
    \[
\cH = \cH'\oplus \C d.
    \]

\subsection{Quantum Heisenberg algebra}
Consider the $\C(q)$-vector space $\C(q)\otimes_\C \cH$, set $H:=1\otimes h$ for any $h\in \h$, $H_{i,r}:=1\otimes h_{i,r}$ for all $i\in I$, and $q^c:= 1\otimes c$. Let ${\widetilde \cH}_q$ be the free associative algebra over $\C(q)$ associated to the vector space $\C(q)\otimes_\C \cH$. The quantum Heisenberg algebra $\cH_q'$ is defined to be ${\widetilde \cH}_q$ modulo the following relations:
    \[
q^{\pm c} \text{ is central},\quad [H_{i,r},H_{j,s}]=\delta_{r,-s}\,\frac{[rA_{i,j}][rc]}{r}
    \]
where $[k] = \frac{q^{k}-q^{-k}}{q - q^{-1}}$ and $[kc] = \frac{q^{kc}-q^{-kc}}{q - q^{-1}}$ for all $k\in \Z$. 

\begin{rem}\label{rem-Heis}
Note that $\cH_q$ is not a true quantization of the algebra $\cH$. More precisely, for any $a\in \Z$ we have an isomorphism of algebras $\cH_q/\langle q^c-q^a \rangle \cong \C(q)\otimes_\C\cH/\langle c-a \rangle$.
\end{rem}

In what follows we consider the algebra 
    \[
\cH_q = \cH_q'\oplus \C q^d\quad \text{where}\quad  q^d H_{i,r}q^{-d}=q^r H_{i,r} \text{ for all }i\in I,\ r\in \Z^\times.
    \]

\subsection{Diagonal modules over quantum Heisenberg algebras}\label{subsec:diagonal_modules}

Assume that $\g$ is not isomorphic to $\fpsl(n,n)$ or $\sl(n,n)$. In these cases $\det A \neq 0$. For each $r\in \Zz$, define the matrix $A(r):=([rA_{i,j}])_{i,j\in I}$. Since $\lim_{q\to 1} [rA_{i,j}]=rA_{i,j}$, we have $\det A(r) \neq 0$ for all $r\in \Zz$. Setting $\phi_{i,r}=H_{i,r}$ for $r>0,\ i\in I$ and using the fact that $\det A\neq 0$, we can find $\phi_{i,r}$ for $r<0$ such that $\{\phi_{i,k}\mid (i,k)\in I\times \Zz\}$ generates $\cH_q$ and
\begin{equation}\label{eq:defining.Heis}
[\phi_{i,r},\phi_{j,s}]=\delta_{i,j}\delta_{r,-s}[rc],\quad \text{for all }\; i, j\in I,\; r, s \in \Zz.
\end{equation}

Since each $\phi_{i,r}$ is a linear combination of the vectors $H_{i,r}$,  $i\in I$, we see that $q^d\phi_{i,r}q^{-d}=q^r\phi_{i,r}$ for all $i\in I, r\in \Zz$. 

We say a $\cH_q$-module $V_q$ is graded if $V_q=\oplus_{r\in \Z}V_{q,r}$, where $V_{q,r}$ is the eigenspace of $q^d$ associated with the eigenvalue $q^r$. A graded $\cH_q$-module $V_q$ is called \emph{diagonal} if $\phi_{i,k}\phi_{i,-k}$ have a common eigenvector (living in $\C(q)$) in $V_q$ for all $(i,k)\in I\times \Zg$. If $V_q$ is a diagonal irreducible $\cH_q$-module, then the elements $\phi_{i,k}\phi_{i,-k}$ are simultaneously diagonalizable on $V_q$ for all $(i,k)\in I\times \Zg$. If $v$ is such an eigenvector with $\phi_{i,k}\phi_{i,-k}v=\lambda_{ik}v$, then observe that 
$$\phi_{i,k}\phi_{i,-k}\phi_{i,\pm k}^p v=(\lambda_{ik} \mp p[kc])\phi_{i,\pm k}^p v, \,\,\,\phi_{i,k}\phi_{i,-k}\phi_{i, r}^p v=\lambda_{ik}\phi_{i, r}^p v.$$ In what follows, we let $\mathscr{K}_{q,a}$ be the category of all diagonal graded $\cH_q$-modules such that $q^c$ acts as $q^a$ with $a\in \Z^\times$.

Fix $a\in \Z^\times$ and $\bs \mu\in \C(q)^{I\times \Z_{>0}}$ such that $\bs\mu(i,k)=\mu_{i,k}$. Define the $\cH_q$-module 
    \[
V_q(\bs \mu,a)=\cH_q/\mathscr{I}_{\bs \mu,a},
    \]
where $\mathscr{I}_{\bs \mu,a}$ is the left ideal of $\cH_q$ generated by $$\{\phi_{i,k}\phi_{i,-k}-\mu_{i,k},\ q^c-q^a,\ q^d\mid (i,k)\in I\times \Zg\}.$$ Let $v_{\bs \mu, a}$ denote the image of $1$ in $V_q(\bs \mu,a)$. In order to construct a basis for $V_q(\bs \mu,a)$, we consider the set $\Phi$ of all functions $\varphi: I\times \Zg \rightarrow \{\pm 1\}$, and the set $\Pi$ of all functions $\pi: I\times \Zg \rightarrow \Zge$ with finite support. Given $(\varphi, \pi)\in \Phi\times \Pi$, define 
\begin{align*}
    \phi_{(\varphi, \pi)}=\prod_{(i,k)\in I\times \Zg} \phi_{i,\varphi(i,k)k}^{\pi(i,k)}\in \cH_q.
\end{align*}
Notice that all factors in $\phi_{(\varphi, \pi)}$ commute among themselves, so the product can be taken in any order.
\begin{lem}\label{lem:strucrure_V_q(mu,a)}
The following statements hold:
\begin{enumerate}
\item The set $\{\phi_{(\varphi, \pi)}v_{\bs \mu, a}\mid (\varphi, \pi)\in \Phi\times \Pi\}$ is a basis of $V_q(\bs \mu,a)$
\item The module $V_q(\bs \mu,a)$ is irreducible if and only if $\mu_{i,k} \notin [ka]\Z$ for all $(i,k)\in I\times \Zg$.
\end{enumerate}
\end{lem}
\begin{proof}
(i): To see that this set generates $V_q(\bs \mu,a)$ we just observe that any monomial that has $\phi_{i,k}$ and $\phi_{i,-k}$ as factors lies in $\mathscr{I}_{\bs \mu,a}$. The linear independence  follows from the fact that the $\phi_{i,k}\phi_{i,-k}$-eigenvalues of each $\phi_{(\varphi, \pi)}v_{\bs \mu, a}$ is uniquely determined by $(\varphi, \pi)$. 

(ii): Notice that 
    \[
\phi_{i,-k}\phi_{i,k}^n v_{\bs \mu, a}=(-n[kc]+\mu_{i,k})\phi_{i,k}^{n-1}v_{\bs \mu, a},\text{ and } \phi_{i,k}\phi_{i,-k}^n v_{\bs \mu, a}=((n-1)[kc]+\mu_{i,k})\phi_{i,-k}^{n-1}v_{\bs \mu, a}.
    \]
Hence, if $\mu_{i,k} = n[ka]$ for some $n\in \Z_{>0}$ (respectively, $n\in \Z_{\leq 0}$), then the vector $\phi_{i,k}^n v_{\bs \mu, a}$ (respectively, $\phi_{i,-k}^{-n+1} v_{\bs \mu, a}$) generates a proper submodule of $V_q(\bs \mu,a)$. The converse is clear.
\end{proof}

\medskip

It follows from Lemma~\ref{lem:strucrure_V_q(mu,a)} that the module $V_q(\bs \mu,a)$ is reducible if and only if $\mathrm{F}_{\bs \mu,a} = \{(i,k)\in I\times \Zg\mid \mu_{i,k}\in [ka]\Z\}\neq \emptyset$. For $\mathrm{F}\subseteq \mathrm{F}_{\bs \mu,a}$, we define $V_q(\bs \mu,a,\mathrm{F})$ be the quotient of $V_q(\bs \mu,a)$ by its submodule generated by $\{(\phi_{i,k})^{\frac{\mu_{i,k}}{[ka]}}v_{\bs\mu, a} \text{ if } \frac{\mu_{i,k}}{[ka]}>0,\ (\phi_{i,-k})^{1-\frac{\mu_{i,k}}{[ka]}}v_{\bs\mu, a} \text{ if }\frac{\mu_{i,k}}{[ka]}\leq 0 \mid  (i,k) \in \mathrm{F}\}$. We continue to denote the image of $v_{\bs \mu, a}$ in $V_q(\bs \mu,a,\mathrm{F})$ by $v_{\bs \mu, a}$. 

To construct a basis for $V_q(\bs \mu,a,\mathrm{F})$ we consider the sets
\begin{align*}
& E_1 = \{ (\phi_{i_1,k_1})^{j_1} \cdots (\phi_{i_l,k_l})^{j_l} v_{\bs\mu, a} \mid (i_j,k_j)\in F,\, 0\leq j_s < \frac{\mu_{i_s,k_s}}{[k_s a]}\}, \\
& E_2 = \{ (\phi_{i_1,-k_1})^{j_1} \cdots (\phi_{i_l,-k_l})^{j_l} v \mid v\in E_1,\ (i_j,k_j)\in F,\, 0\leq j_s \leq -\frac{\mu_{i_s,k_s}}{[k_s a]}\}.
\end{align*}
Next, we let $\Phi(\mathrm{F}) \subseteq \Phi$ denote the set of all functions $\varphi \in \Phi$ for which the following hold
\begin{align*}
    \varphi(i,k)=\begin{cases} -1, \quad \text{if } \frac{\mu_{i,k}}{[ka]}> 0;\\
     1, \quad \text{if } \frac{\mu_{i,k}}{[ka]}\leq 0,
    \end{cases}
\end{align*}
for all $(i,k) \in \mathrm{F}$. For all $v\in E_2$ we define $\Phi(v) = \{\varphi\in \Phi(F)\mid \phi_{i,-\varphi(i,k)k}\text{ is not a factor occurring in }v\}$, then
    \[
\{\phi_{(\varphi, \pi)}v \mid v\in E_2,\ (\varphi, \pi)\in \Phi(v)\times \Pi\}
    \]
is a basis of $V_q(\bs \mu,a,\mathrm{F})$.

\begin{lem}\label{lem:H_{j,+-k}v_{mu, a}.not.zero}
The vectors $H_{j,k}v_{\bs \mu, a}$ and $H_{j,-k}v_{\bs \mu, a}$ cannot vanish simultaneously for all $j\in I$.
\end{lem}
\begin{proof}
If $H_{j,k}v_{\bs \mu, a}=0$, then $\phi_{j,k}v_{\bs \mu, a}=0$. Moreover, if $H_{j,-k}v_{\bs \mu, a}=0$, then $\sum_{i\in I} b_{i,k}\phi_{i,-k}v_{\bs \mu, a}=0$. This  together with the linearly independence of the vectors $\phi_{i,-k}v_{\bs \mu, a}$ (they have different $\phi_{i,k}\phi_{i,-k}$-eigenvalues) imply that $\phi_{i,-k}v_{\bs \mu, a}=0$ for all $i\in I$. But then we have $\phi_{j,\pm k}v_{\bs \mu, a}=0$, contradicting the fact that $\phi_{i,\varphi(i,k)k}v_{\bs \mu, a}\neq 0$ for any $i\in I$, $k\in \Zz$.
\end{proof}

It is worth to notice that Lemma~\ref{lem:H_{j,+-k}v_{mu, a}.not.zero} in the non-super setting is a consequence of the fact that $A_{j,j}\neq 0$ for all $j\in I$. Such argument clearly does not work for isotropic roots.

Finally, we define $\Phi_{\bs \mu,a}=\Phi(\mathrm{F}_{\bs \mu,a})$, and we write $\varphi=\varphi_{\bs \mu, a}$ to indicate that a function $\varphi$ lies in $\Phi_{\bs \mu,a}$. Note that $V_q(\bs \mu, a, \mathrm{F}_{\bs \mu,a})$ is irreducible and that $V_q(\bs \mu, a,\emptyset) = V_q(\bs \mu,a)$. It is obvious that the modules of the form $V_q(\bs \mu,a,\mathrm{F})$ and $V_q(\bs \nu,a,\mathrm{F}')$ are not isomorphic if $\mathrm{F}_{\bs \mu,a}\neq \mathrm{F}_{\bs \nu,a}$ or $\mathrm{F}\neq \mathrm{F}'$. Moreover we have the following result:

\begin{prop}
Let $\mathrm{F}\subseteq \mathrm{F}_{\bs \mu,a}=\mathrm{F}_{\bs \nu,a}$. Then the modules $V_q(\bs \mu,a,\mathrm{F})$ and $V_q(\bs \nu,a,\mathrm{F})$ are isomorphic if and only if $\mu_{i,k} = \nu_{i,k}$ for all but finitely many indices, $\mu_{i,k}-\nu_{i,k}\in [ka]\Z$ for all $(i,k)\in I\times \Zg$, and $\varphi_{\bs \mu, a}(i,k)=\varphi_{\bs \nu, a}(i,k)$ for all $(i,k)\in \mathrm{F}_{\bs \mu,a}$.
\end{prop}
\begin{proof}Fix $(j,\ell) \in I\times \Zg$ and let $\bs \kappa=(\kappa_{i,k}) \in \C(q)^{I\times \Zg}$ be defined as $\kappa_{j, \ell}=\mu_{j,\ell}+[\ell a]$ and $\kappa_{i,k}=\mu_{i,k}$ if $ (i,k)\neq (j,\ell)$. Then, we have a homomorphism of $\cH_q$-modules $\xi_{j,\ell}:V_q(\bs \mu,a)\rightarrow V_q(\bs \kappa, a)$ which maps $v_{\bs \mu, a}$ to $\phi_{j,\ell} v_{\bs \kappa, a}$. Moreover, $\xi_{j,\ell}$ is an isomorphism if and only if $\mu_{j,\ell}\neq 0$ which is equivalent to ask that $\varphi_{\bs \mu, a}(i,k)=\varphi_{\bs \kappa, a}(i,k)$ for all $(i,k)\in \mathrm{F}_{\bs \mu,a}$. Observe that, in this case the isomorphism $\xi_{j,\ell}^{-1}$ maps $v_{\bs \kappa, a}$ to $\phi_{j,-\ell} v_{\bs \mu, a}$. Furthermore, the isomorphism $\xi_{j,\ell}$ maps the kernel of the canonical map $V_q(\bs \mu,a) \rightarrow V_q(\bs \mu,a,\mathrm{F})$ to the kernel of the canonical map $V_q(\bs \kappa,a) \rightarrow V_q(\bs \kappa,a,\mathrm{F})$ for any $\mathrm{F}\subseteq \mathrm{F}_{\bs \mu,a}$. Hence, it induces an isomorphism between $V_q(\bs \mu,a,\mathrm{F})$ and $V_q(\bs \kappa,a,\mathrm{F})$.

Finally, if $V_q(\bs \mu,a,\mathrm{F})$ and $V_q(\bs \nu,a,\mathrm{F})$ satisfy the hypothesis of the proposition, then the desirable isomorphism between such modules is given by composition of finitely many suitable powers of $\xi_{j,\ell}$, or its inverse, for all $(j,\ell)\in I\times \Zg$ such that $\mu_{j,\ell} \neq  \nu_{j,\ell}$.
\end{proof}

\subsection{Some remarks} Each $\varphi \in \Phi$ induces a triangular decomposition of $\cH_q$
\begin{align*}
    \cH_q=\cH_q^{-\varphi}\otimes \C(q)[q^{\pm c}] \otimes \cH_q^{\varphi},
\end{align*}
where $\cH_q^{\pm\varphi}$ is the subalgebra of $\cH_q$ generated by $\phi_{i,\pm\varphi(i,r)r}$, $(i,r)\in I\times \Zg$. The \emph{$\varphi$-Verma module} of level $a$ is the $\cH_q$-module $V^\varphi(a)$ generated by a nonzero vector $v_{\varphi,a}$ such that $q^c v_{\varphi,a}=q^a v_{\varphi,a}$ and $\phi_{i,\varphi(i,r)r}v_{\varphi,a}=0$ for all $(i,r)\in I\times \Zg$. As vector spaces, the module $V^\varphi(a)$ is isomorphic to $\cH_q^{-\varphi}$. We point out that when $\frac{\mu_{i,k}}{[ka]}\in \{0,1\}$ for all $(i,k)\in I\times \Zg$, then $\Phi_{\mu,a}=\{\psi\}$ and $V_q(\bs \mu, a,\mathrm{F}_{\bs \mu,a})$ is isomorphic to the \emph{${\varphi}$-Verma module} $V^\varphi(a)$ with $\varphi=-\psi$.

The following subalgebra of $\C(q)$ will be important in the subsequent sections:
    \[
\A = \{\frac{f(q)}{g(q)}\in \C(q)\mid g(1)\neq 0\}\subseteq \C(q).
    \]

\begin{rem}\label{rem:V_q.to_V}
\begin{enumerate}
\item The construction of classical analogs of modules of the form $V_q(\bs\mu, a, F)$ also works for the classical Heisenberg algebra $\cH$ for all $a\in \C^\times$. The only difference is that the eigenvalues $\mu_{i,k}$ should live in $\C$ instead of $\C(q)$. Any $\bs \mu\in \A^{I\times \Z_{>0}}$ induces $\bs \mu(1) \in \C^{I\times \Z_{>0}}$ where $\bs \mu(1) (i,k) = \bs \mu_{i,k}(1)$, and we denote the $\cH$-module $V(\bs\mu(1), a, \mathrm{F})$ also by $V(\bs\mu, a, \mathrm{F})$ (we just drop the $q$ from the notation). We denote by $L(\bs\mu, a)$ the irreducible quotient of $V(\bs\mu, a)$. We notice however that $V(\bs \mu, a,\Phi_{\bs \mu,a})$ is reducible in general (in fact, if some $\mu_{i,k}(1)\in a\Z$,  but $\mu_{i,k}\not\in [ka]\Z$, then $V(\bs \mu, a,\Phi_{\bs \mu,a})$ is not an irreducible $\cH$-module) and $L(\bs\mu, a)$ is its irreducible quotient. 

\item  The category $\mathscr{K}_a$ of diagonal $\cH$-modules of level $a$ was studied in \cite{BBFK, FGM14}. A non-trivial irreducible $\cH$-module in $\mathscr{K}_{a}$ has finite-dimensional graded subspaces if and only if it is isomorphic to a ${\varphi}$-Verma module with ${\varphi}$ constant \cite{BBFK}. Moreover, it was proved in \cite{FGM14} that any simple module in $\mathscr{K}_a$ is isomorphic to some $L(\bs \mu, a)$.
\end{enumerate}
\end{rem}

\begin{rem}\label{rem:H_{j,+-k}v_{mu, a}.not.zero}
\begin{enumerate}
\item If $\g=\fpsl(n,n)$, then the elements $h_i\in \h$ with $i\in I$ are no longer linearly independent. However, each $h_j$ is a linear combination of $h_i$ for $i\in I\setminus\{j\}$, and hence, the action of $H_{|\Sigma|, r}$ on any $\cH_q$-module is determined by the action of the other generators $H_{i,r}$ for $i\in I \setminus\{|\Sigma|\}$. In this case, we consider $I'= \{1,\ldots, |\Sigma|-1\}$ and we get a linearly independent set $\{h_i\mid i\in I'\}$ which is a basis of $\h$. Now all constructions of this subsection will follow by replacing $I$ by $I'$ everywhere and $A$ by the matrix $A'$ obtained by deleting the latter row and column of $A$. We also point out that Lemma~\ref{lem:H_{j,+-k}v_{mu, a}.not.zero} still holds for all $j\in I$ in this case (just replace $I$ by $I'$ in the proof).

\item If $\g=\sl(n,n)$, then there is a linear combination of the vectors $h_i$, $i\in I$ resulting in the identity matrix $z := I_{2n,2n}\in \h$. Then, for all $r\in \Z^\times$, the vector $Z_{r}=1\otimes z_r$ is central in $\cH_q'$, and it must act trivially on any simple graded $\cH_q$-module. In other words, any simple $\cH_q$-module is just the pull back (under the map induced by the canonical projection $\sl(n,n)\to \fpsl(n,n)$) of a simple module over the $\fpsl$-type Heisenberg algebra. Motivated by this, the $\cH_q$-modules we will consider in this paper are precisely those obtained via such pull back from diagonal modules over the $\fpsl$-type Heisenberg algebra. This implies that the vectors $Z_{r}=1\otimes z_r$ act trivially on any such module, and thus the action of $H_{|\Sigma|, r}$ is determined by the action of the other generators $H_{i,r}$, $i\in I'$. Hence, as in the $\fpsl$ case, all constructions of this subsection will follow by replacing $I$ by $I'$ and $A$ by $A'$ everywhere. Also similarly to the $\fpsl$ case, we have that Lemma~\ref{lem:H_{j,+-k}v_{mu, a}.not.zero} still holds for all $j\in I$.
\end{enumerate}
\end{rem}

\section{Induced \texorpdfstring{$\gh$}{g hat}-modules}

In this section we consider modules over $\gh$ which are induced from modules over $\cH$. For this we take a system of simple roots $\Sigma=\{\alpha_1,\ldots, \alpha_r\}$ of $\D$, and we recall that for any subset $X\subset \Sigma$ we can define $\D(X)=\Z X\cap \D$ and $\D(X)^+ = \D(X)\cap \D^+$. Then the set
\begin{align*}
\hD \supseteq P(\Sigma,X) & = \{\alpha+n\delta \mid \alpha \in\D^+\setminus \D(X)^+, n\in \Z\} \\
& \cup \{\alpha+n\delta\mid \alpha\in \D(X)\cup \{0\}, n\in \Z_{>0}\}\cup \D(X)^+
\end{align*}
is a parabolic partition of $\hD$. Moreover, any parabolic partition $P$ of $\hD$ is of this form for some choice of $\Sigma$ and of $X\subseteq \Sigma$. Borel subalgebras of $\gh$ are in bijection with parabolic partitions of $\hD$; namely, for any given parabolic partition $P$ we can define the Borel subalgebra $\B = \hh\oplus \bigoplus_{\alpha\in P}\gh_\alpha$. 

In what follows we will be mainly interested in the so called natural Borel subalgebras. These are given by choosing $X=\emptyset$. Notice that in this case $\B=\hh\oplus (\h\otimes t\C[t])\oplus (\n^+\otimes \C[t,t^{-1}])$. Now we define the following subalgebras of $\gh$:
    \[
\cU^\pm = \n^\pm\otimes \C[t,t^{-1}],\quad \cN^\pm = (\h\otimes t\C[t])\oplus \cU^\pm,\quad \cP = \h\oplus \cH \oplus \cU^+.
    \]
The subalgebra $\cP$ is a minimal parabolic subalgebra of $\gh$ containing $\B + \cH$. The Heisenberg algebra $\cH$ is called the Levi factor of $\cP$.

Let $V$ be a $\cH$-module in $\mathscr{K}_a$ and let $\lambda \in \hh^*$ such that $\lambda(c)=a$. We extend $V$ to a $\cP$-module by setting $hv=\lambda(h)v$ for all $v\in V$ and $h \in \h$, and by letting $\cU^+$ act trivially on it. Consider the induced $\Ugh$-module
\begin{align*}
    M(\lambda,V)=\Ind({\cP},{\Ugh};V)=\Ugh\otimes_{U(\cP)}V.
\end{align*}
When $V$ is simple, $M(\lambda,V)$ has a unique maximal module and hence a unique simple quotient. The $\Ugh$-module $ M(\lambda,V)$ is called the \emph{generalized imaginary Verma module}  associated with $\cP$, $V$ and $\lambda$. If $V$ is a $\varphi$-Verma module of $\cH$, then $M^\varphi(\lambda):=M(\lambda,V)$ is called the \emph{$\varphi$-imaginary Verma module} of weight $\lambda$. If $\varphi\equiv 1$, we write $M(\lambda)=M^\varphi(\lambda)$, and we have the following result from \cite{CF}.

\begin{thm}[\cite{CF}]\label{thm:simplicity_criteria_standard_imaginary}
$M(\lambda)$ is irreducible if and only if $\lambda(c)\neq 0$.
\end{thm}

In the next section we will generalize Theorem~\ref{thm:simplicity_criteria_standard_imaginary} to the case where $V$ is an arbitrary irreducible module in $\mathscr{K}_a$.

\medskip

\subsection{Irreducibility criterion}\label{sec:irr.criteria.classical} 

Let $\bs \mu \in \C^{I\times \Zg}$, $\lambda \in \h^*$ with $\lambda(c)=a\neq 0$, and $\mathrm{F}\subseteq \mathrm{F}_{\bs \mu,a}$. Take the $\cH$-module $V=V(\bs \mu,a, F)$ as defined in Subsection~\ref{subsec:diagonal_modules} and consider the associated generalized imaginary Verma $\Ugh$-module $M(\lambda,V)$. For simplicity, write $v_j=\phi_{(\varphi, \pi)}v_{\bs \mu, a}$ for $j=(\varphi, \pi) \in \Phi\times \Pi$, and  denote by $d_j$ the degree of the vector $v_j$. 

Consider the total order ``$\leq$'' on $\D^+$ defined as follows: for $i,j \in I$, set $\alpha_i\leq \al_j$ if $i\leq j$, then extend the ordering to $\D^+$ by setting $\alpha\leq \beta \leq \gamma$ if $\beta = \alpha+\gamma$ and $\alpha\leq \gamma$. This induces a total order on $\D^+\times \Z$ given by
\begin{equation}\label{order.DZ}
    (\alpha,r)\leq (\alpha', r')\; \text{if and only if}\; \alpha< \alpha' \; \text{or}\; \alpha= \alpha', r \leq r'.
\end{equation}
Consider now the set $\mathsf{M}$ equipped with the lexicographical order induced by the total order on $\D^+\times \Z$; that is, given $\mathsf{m}_1, \mathsf{m}_2 \in \mathsf{M}$, let $(\alpha_0,k_0)=\min\{(\alpha,k)\in \D^+\times \Z\,\mid\, \mathsf{m}_1(\alpha,k)\neq \mathsf{m}_2(\alpha,k)\}$, then
\begin{equation}\label{order.M}
    \mathsf{m}_1<\mathsf{m}_2,\; \text{if }\; \mathsf{m}_1(\alpha_0,k_0)<\mathsf{m}_2(\alpha_0,k_0). 
\end{equation}
Now, this order on $\mathsf{M}$ induces in the obvious way a total order on the set of all monomials
\begin{align*}
    &x^+_\mathsf{m}:=\prod_{(\alpha,r)\in \D^+\times \Z}^{\rightarrow} (x_{\alpha,r})^{\mathsf{m}(\alpha,r)},  &&x^-_\mathsf{m}:=\prod_{(\alpha,r)\in \D^+\times \Z}^{\leftarrow} (x_{-\alpha,r})^{\mathsf{m}(\alpha,r)},
\end{align*}
where the $\rightarrow$  over the product indicates that the product is written in increasing order from left to right with respect to the order on $\D^+\times \Z$, and the opposite for $\leftarrow$.

Since we have an isomorphism of $\hh$-modules $M(\lambda, V)\cong U(\cU^-)\otimes V$, the set of monomials 
$$\{x^-_{\mathsf{m}}v_j\,\mid\, \mathsf{m}\in \mathsf{M},\; j\in \Phi(F)\times \Pi\}$$ is a basis of $M(\lambda, V)$ consisting of weight vectors. Define $Q = \Z\Sigma$ and $Q^+=\Z_{\geq 0}\Sigma$, and let $v\in M(\lambda, V)_\mu$ be a nonzero vector of weight $\mu=\lambda-\beta +n\delta$, for some $\beta \in Q^+\setminus\{0\}$, $n\in \Z$. Since $\beta \in Q^+\setminus\{0\}$, there exist unique $a_i\in \Z_{\geq 0}$ such that $\beta = \sum_{i=1}^{N-1}a_i\alpha_i$. Define 
\begin{align*}
    \htt{v}:=\htt{\beta}=\sum_{i=1}^{N-1}a_i.
\end{align*}

Although there are some versions of the next result, none of them consider the order we fixed on the monomials of $U(\gh)$. Such order will be crucial for the next section where we study quantum deformations of generalized imaginary Verma modules.

\begin{lem}\label{lem.ht2}
Let $\bs \mu \in \C^{I\times \Zg}$, $\lambda \in \h^*$ with $\lambda(c)=a\neq 0$, and $\mathrm{F}\subseteq \mathrm{F}_{\bs \mu,a}$. Let $V$ be an $\cH$-module of the form $V(\bs\mu, a, F)$ or $L(\bs\mu,a)$. If $v\in M(\lambda, V)$ is a nonzero weight vector such that $\htt{v}> 0$, then there exists $y \in \Ugh$ such that $yv\neq 0$ and $\htt{yv}<\htt{v}$.
\begin{proof}
Write
\begin{equation}\label{v.pbw}
    v=\sum_{j\in L}a_j x^-_{\mathsf{m}_j} v_j,
\end{equation}
for some finite set $L$ and $a_j\in \C^\times$.

Let $\mathsf{m}_{j_0}=\min \{\mathsf{m}_j\,\mid\, j\in L\}$ with respect to the order \eqref{order.M}, and let $(\alpha_0,r_0)=\min \{(\alpha,r)\in \D^+\times \Z\,\mid\, \mathsf{m}_{j_0}(\alpha,r)>0\}$
with respect to the order \eqref{order.DZ}.  In other words, $\mathsf{m}_{j_0}$ is the smallest element of $\mathsf{M}$ such that $x^-_{\mathsf{m}_{j_0}}$ is a nontrivial summand in the expression \eqref{v.pbw}, and $(x^-_{\alpha_0,r_0})^{\mathsf{m}_{j_0}(\alpha_0,r_0)}$ is its right most factor. 
 
Let $\mathsf{m}'$ be the element of $\mathsf{M}$ such that $\mathsf{m}'(\alpha,r)=\mathsf{m}_{j_0}(\alpha,r)$ if $(\alpha,r)\neq (\alpha_0,r_0)$ and $\mathsf{m}'(\alpha_0,r_0)=\mathsf{m}_{j_0}(\alpha_0,r_0)-1$. Let $\ell \in \Zz$ be such that neither $\phi_{i,\ell+r_0}$ or $\phi_{i,-\ell+r_0}$ is a factor in any $v_j$ for all $j\in L$ and $i\in I$. Then, since $V\cong V(\bs \mu,a,\mathrm{F})$, by Lemma~\ref{lem:H_{j,+-k}v_{mu, a}.not.zero} (see also Remark~\ref{rem:H_{j,+-k}v_{mu, a}.not.zero}), we may choose the sign of $\ell$ so that $\phi_{i,\ell+r_0}v_{j_0}\neq 0$, and hence that $[x^+_{\alpha_0,\ell},x^-_{\alpha_0,r_0}]v_{j_0} = H_{\alpha_0,\ell+r_0}v_{j_o}\neq 0$. Since $x^+_{\alpha_0,\ell}v_j=0$ for all $j\in J$, we obtain
\begin{equation*}
    x^+_{\alpha_0,\ell}v=\sum_{j\in L}a_j [x^+_{\alpha_0,\ell},x^-_{\mathsf{m}_j}] v_j.
\end{equation*}
Now, we claim that there is only one occurrence of $x^-_{\mathsf{m}'}h_{\alpha_0,\ell+r_0}v_{j_0}$ in the basis expansion of $x^+_{\alpha_0,\ell}v$. 

It follows from the order we fixed on the monomials of $U(\gh)$ that the only way to have a nonzero bracket of elements of $\gh$ in the Leibniz expansion of $[x^+_{\alpha_0,\ell},x^-_{\mathsf{m}_j}]$ is if $x^-_{\mathsf{m}_j}$ has factors of the form: (1) $x^-_{\alpha_0,k}$; (2) $x^-_{\alpha_0 - \beta,k}$ for $\beta < \alpha_0$; or (3) $x^-_{\alpha_0+\beta,k}$ for $\alpha_0< \alpha_0+\beta$. For each factor in case (3) we obtain an element in $U(\cU^-)$, and hence with no factor in $\cH$. In the case (2), for each factor $x^-_{\alpha_0 - \beta,k}$ with $\beta < \alpha_0$ we get a factor of the form $x^+_{\beta,k+\ell}$. Since $\alpha_0$ is minimal and $\beta<\alpha_0$, in the process of commuting $x^+_{\beta,k+\ell}$ it with the remaining factors of $x^-_{\mathsf{m}_j}$ we can only produce elements of the form $x^+_{\gamma,p}$ with $\gamma<\alpha_0$ or elements in $\cU^-$, and again we get elements with no factor in $\cH$.

Finally, writing $x^-_{\mathsf{m}_j} = u_{\mathsf{m_j}}'u_{\mathsf{m_j}}$ where $u_{\mathsf{m_j}}'\in U(\cU^-)$ does not have any factor of the form $x^-_{\alpha_0, *}$, we conclude that we just have to look at $u_{\mathsf{m_j}}'[x^+_{\alpha_0,\ell},u_{\mathsf{m_j}}]$. Write $u_{\mathsf{m_j}} = (x^-_{\alpha_0,r_{j_1}})^{t_{j_1}}\cdots (x^-_{\alpha_0,r_{j_k}})^{t_{j_k}}$, let $t=t_{j_1}+\cdots + t_{j_k}$, and notice that
\begin{align*}
& u_{\mathsf{m_j}}'[x^+_{\alpha_0,\ell},u_{\mathsf{m_j}}] = u_{\mathsf{m_j}}'\sum_i^{t} x^-_{\alpha_0,r_{j_1}}\cdots h_{\alpha_0,\ell+r_{j_i}}\cdots x^-_{\alpha_0,r_{j_k}} \\
& = u_{\mathsf{m_j}}'\sum_i^{t} x^-_{\alpha_0,r_{j_1}}\cdots [h_{\alpha_0,\ell+r_{j_i}}, x^-_{\alpha_0,r_{j_i}}\cdots x^-_{\alpha_0,r_{j_k}}] + (x^-_{\alpha_0,r_{j_1}})^{t_{j_1}}\cdots(x^-_{\alpha_0,r_{j_i}})^{t_{j_i}-1}\cdots (x^-_{\alpha_0,r_{j_k}})^{t_{j_k}}h_{\alpha_0,\ell+r_{j_i}} \\
& = u_{\mathsf{m_j}}' \sum_i^{t} (x^-_{\alpha_0,r_{j_1}})^{t_{j_1}}\cdots(x^-_{\alpha_0,r_{j_i}})^{t_{j_i}-1}\cdots (x^-_{\alpha_0,r_{j_k}})^{t_{j_k}}h_{\alpha_0,\ell+r_{j_i}} \mod U(\cU^-).
\end{align*}
Hence, the only way to have $$u_{\mathsf{m_j}}'(x^-_{\alpha_0,r_{j_1}})^{t_{j_1}}\cdots(x^-_{\alpha_0,r_{j_i}})^{t_{j_i}-1}\cdots (x^-_{\alpha_0,r_{j_k}})^{t_{j_k}}h_{\alpha_0,\ell+r_{j_i}} = x^-_{\mathsf{m}'}h_{\alpha_0,\ell+r_0}$$ is if $$u_{\mathsf{m_j}}'=u_{\mathsf{m_{j_0}}}',  \,\,\,\, (x^-_{\alpha_0,r_{j_1}})^{t_{j_1}}\cdots(x^-_{\alpha_0,r_{j_i}})^{t_{j_i}-1}\cdots (x^-_{\alpha_0,r_{j_k}})^{t_{j_k}} = (x^-_{\alpha_0,r_{{j_0}_1}})^{t_{{j_0}_1}}\cdots (x^-_{\alpha_0,r_0})^{\mathsf{m}_{j_0}(\alpha_0,r_0)-1},$$ and $h_{\alpha_0,\ell+r_{j_i}} = h_{\alpha_0,\ell+r_0}$. But this implies that only one occurrence of $x^-_{\mathsf{m}'}h_{\alpha_0,\ell+r_0}v_{j_0}$ in the basis expansion of $x^+_{\alpha_0,\ell}v$ comes from $[x^+_{\alpha_0,\ell},x^-_{\mathsf{m}_{j_0}}] v_{j_0}$. This shows that $x^+_{\alpha_0+\ell\delta}v\neq 0$ and since $\htt{x^+_{\alpha_0,\ell}v}=\htt{v}-\htt{\alpha_0}<\htt{v}$, the statement follows.
\end{proof}
\end{lem}

\medskip
The next result generalizes the case $X=\emptyset$ of \cite[Theorem~5.9]{CF} to the case where we let $V$ be an arbitrary irreducible diagonal module. In the non-super case this was shown in \cite{BBFK}.

\begin{thm}\label{thm:main_1}
Let $\bs \mu \in \C^{I\times \Zg}$, $\lambda \in \h^*$ with $\lambda(c)=a\neq 0$, and $\mathrm{F}\subseteq \mathrm{F}_{\bs \mu,a}$. Let $V$ be an $\cH$-module of the form $V(\bs\mu, a, F)$ or $L(\bs\mu,a)$. For any nonzero submodule $N\subseteq M(\lambda,V)$, we have $N\cap (1\otimes V)\neq 0$. In particular, $M(\lambda,V)$ is irreducible if and only if $V$ is irreducible.
\end{thm}
\begin{proof}
Let $0\neq v\in N$. If $v\in 1\otimes V$, we are done. Suppose $v\not\in 1\otimes V$. Then $\htt{v}>0$, and Lemma~\ref{lem.ht2} allows us to use induction on $\htt{v}$ to prove the statement.
\end{proof}

\subsection{Kac modules} Recall that if $\g$ is one of Lie superalgebras $\gl(m,n)$, $\sl(m,n)$, $\fpsl(n,n)$ or $\osp(2,2m)$ for $m>1$, then we have a $\Z$-grading $\g = \g_{-1}\oplus \g_0\oplus \g_1$ which is consistent with the $\Z_2$-grading of $\g$. Such decomposition induces a similar decomposition of the affine Kac-Moody Lie superalgebra $\gh = \gh_{-1}\oplus \gh_0\oplus \gh_1$. Set $\gh_+=\gh_0\oplus \gh_1$. It is well known that there is a distinguished set of simple roots $\Sigma$ for which $\n^\pm = \n_0^\pm\oplus \g_{\pm 1}$. Fix such $\Sigma$ and consider the corresponding parabolic subalgebra $\cP_{\Sigma} = \h\oplus \cH \oplus \n^+ \otimes \C[t,t^{-1}]$. We have that
\begin{equation}\label{eq:inclusions}
\gh_+ = \n_0^-\otimes \C[t,t^{-1}]\oplus \cP_{\Sigma}\quad \text{and}\quad \gh\supset \gh_+ \supset \cP_{\Sigma}.
\end{equation}

For any $\gh_0$-module $L$ we can now define the  \emph{Kac module}
    \[
K(L) = U(\gh)\otimes_{U(\gh_+)} L,
    \]
where we let $\gh_1$ acts trivially on $L$. 

Notice that the Heisenberg algebra $\cH$ coincides with the imaginary subalgebra of $\gh_0$. Thus, for any $\cH$-module $V$ and $\lambda \in \h^*$ we can consider the $\gh_0$-module
    \[
    M_0(\lambda,V)=\Ind({\cP_0},U(\gh_0);V)=U(\gh_0)\otimes_{U(\cP_0)}V,
    \]
where $\cP_0=\cP\cap \gh_0 = \h\oplus \cH\oplus \n_0^+ \otimes \C[t,t^{-1}]$. We also set $M_{\Sigma}(\lambda, V)= \Ind({\cP_{\Sigma}},U(\gh);V)$. 

\medskip

\begin{thm}\label{thm-Kac} The following statements hold:
\begin{enumerate}
\item For any $\cH$-module $V$, we have an isomorphism of $\gh$-modules $M_{\Sigma}(\lambda, V)\cong K(M_0(\lambda,V))$.
\item Let $\bs \mu \in \C^{I\times \Zg}$, $\lambda \in \h^*$ with $\lambda(c)=a\neq 0$, and $\mathrm{F}\subseteq \mathrm{F}_{\bs \mu,a}$. Let $V$ be an $\cH$-module of the form $V(\bs\mu, a, F)$ or $L(\bs\mu,a)$. Then the Kac module $K(M_0(\lambda,V))$ is irreducible if and only if $M_0(\lambda,V)$ is irreducible.
\end{enumerate}
\end{thm}
\begin{proof}
Part (i): The inclusions \eqref{eq:inclusions} imply that the following isomorphim of $\gh$-modules holds:
    \[
M_\Sigma(\lambda, V) = U(\gh)\otimes_{U(\cP_\Sigma)}V \cong U(\gh)\otimes_{U(\gh_+)}\left(U(\gh_+)\otimes_{U(\cP_\Sigma)}V\right).
    \]

We claim now that the $\gh_+$-modules $U(\gh_+)\otimes_{U(\cP_\Sigma)}V$ and $M_0(\lambda, V)$ are isomorphic. It follows from their definition that both modules are isomorphic as $\gh_0$-modules. Indeed, both modules are isomorphic to $U(\n_0^-\otimes \C[t,t^{-1}])\otimes_\C V$ as vector spaces, and the $\gh_0$-action on them given by the left multiplication of $U(\gh_0)$. Now the claim follows because $\gh_1$ acts trivially on both modules. Hence, we have isomorphims of $\gh$-modules
\begin{align*}
M_{\Sigma}(\lambda, V) & \cong U(\gh)\otimes_{U(\gh_+)}\left(U(\gh_+)\otimes_{U(\cP_\Sigma)}V\right) \\
& \cong U(\gh)\otimes_{U(\gh_+)}M_0(\lambda, V) \\
& =  K(M_0(\lambda, V)).
\end{align*}

Part (ii): Note that $M_0(\lambda,V)$ is irreducible if and only if $V$ is irreducible by Theorem~\ref{thm:main_1}. The statement follows from part~(i) and Theorem~\ref{thm:main_1}.
\end{proof}

\medskip

\section{Generalized imaginary Verma modules for \texorpdfstring{$U_q\slhmn$}{Uq  sl(m|n) hat}}

Throughout this section we assume that $\g$ is a Lie superalgebra of type $\sl(m,n)$. Let $N=n+m$ and recall that for any given $N$-tuple $\s := (s_1,\ldots, s_N)\in \{\pm 1\}^N$ we have a corresponding set of simple roots $\Sigma_\s$ and its corresponding Cartan matrix $A^\s$.  From now on we fix  $\s$ and we set $\Sigma = \Sigma_\s$ and $A = A^\s$. Set also $Q = \Z\Sigma$ the root lattice corresponding to $\Sigma$. Our goal is to define analogs of generalized imaginary Verma modules for the quantum superalgebra $\Uqgh$ and to prove that such modules are quantum deformations of generalized imaginary Verma modules over $\gh$.

\medskip

\subsection{\texorpdfstring{$\Uqgh$}{Uq g hat} and its \texorpdfstring{$\A$}{A}-form} Let $q$ be an indeterminate. In the {\it new Drinfeld realization}, the quantum affine superalgebra  $\Uqgh$ is the associative superalgebra over $\C(q)$ generated by vectors $X^\pm_{i,n}, H_{i,r}$, $K^{\pm 1}_i, q^{\pm c}, q^{\pm d}$,  $i \in I := \{1,\ldots, |\Sigma|\},\; n\in \Z, \; r\in\Zz$, where their parities are given by $|X^\pm_{i,r}|=|i|=(1-s_is_{i+1})/2$, and all remaining generators are even. The defining relations of $\Uqgh$ are given as follows:
\begin{align*}
&q^{\pm c} \text{ is central},\quad K_iK_j=K_jK_i,\quad K_iX^\pm_j(z)K_i^{-1}=q^{\pm A_{i,j}}X^\pm_j(z),\\
&q^dX^\pm_{i,r}q^{-d}=q^rX^\pm_{i,r},\quad q^dH_{i,r}q^{-d}=q^rH_{i,r}, \quad  q^dK_iq^{-d}=K_{i}, \\
&[H_{i,r},H_{j,s}]=\delta_{r+s,0}\,\frac{[rA_{i,j}]}{r}\frac{q^{rc}-q^{-rc}}{q-q^{-1}},\\
&[H_{i,r},X^{\pm}_j(z)]=\pm\frac{[rA_{i,j}]}{r}q^{-c(r\pm|r|)/2}z^rX^\pm_j(z),\\
&[X^+_i(z),X^-_j(w)]=\frac{\delta_{i,j}}{q-q^{-1}}\Bigl(\delta\bigl(q^{c}\frac{w}{z}\bigr)K_i^+(w)-\delta\bigl(q^c\frac{z}{w}\bigr)K_i^-(z)\Bigr),\\
& (z-q^{\pm A_{i,j}}w)X^\pm_i(z)X^\pm_j(w)+(-1)^{|i||j|}(w-q^{\pm A_{i,j}}z)X^\pm_j(w)X^\pm_i(z)=0 & (A_{i,j}\neq 0), \\
& [X^\pm_i(z),X^\pm_j(w)]=0
& (A_{i,j}=0), \\
& \Sym_{z_1,z_2}\lb{X^\pm_i(z_1),\lb{X^\pm_i(z_2),X^\pm_{i\pm 1}(w)}}=0\,
& (A_{i,i}\neq 0,\ i\pm1\in I), \\
& \Sym_{{z_1,z_2}}\lb{X^\pm_i(z_1),\lb{X^\pm_{i+ 1}(y),\lb{X^\pm_i(z_2),X^\pm_{i- 1}(w)}}}=0
& (A_{i,i}=0,\ i\pm 1 \in I),   
\end{align*}
where $z,w,z_1,z_2$ are formal commutative variables,
    \begin{align*}
&X^\pm_i(z) := \sum_{k\in \Z}X^\pm_{i,k}z^{-k}, \qquad  \delta(z):=\sum_{k\in \Z}z^{k}\\
&K_i^\pm(z) := \sum_{r\in \Z_{\geq 0}}K^\pm_{i,\pm r}z^{\mp r} = K_i^{\pm 1}\exp \left(\pm (q-q^{-1})\sum_{r>0}H_{i,\pm r}z^{\mp r}\right),
        \end{align*}
and 
    \[
\lb{X,Y} = [X,Y]_{q^{-\beta(h_\gamma)}} := XY - (-1)^{|X||Y|}q^{-\beta(h_\gamma)}YX
    \]
if $X$ and $Y$ have weights $\beta,\gamma\in Q$, respectively (i.e. $K_iXK_i^{-1}=q^{\beta(h_i)}X$ and $K_iYK_i^{-1}=q^{\gamma(h_i)}Y$). 

Notice that the defining relations depend on the choice of  $\Sigma$, however, it is known that the superalgebra $\Uqgh$ is independent of such choice. Finally, we define the following Lie subalgebras of $\Uqgh$:
\begin{itemize}
    \item[] $\cH^{\pm}_q$ the subalgebra generated by $H_{i,\pm r}$, $i\in 
    I, r\in \Z_{>0}$;
    \item[] $\cH_q$ the subalgebra generated by $\cH^{\pm}_q$, $q^d$ and $q^c$;
    \item[] $\Uqhh$ the subalgebra generated by $K_i$, $i\in I$, $q^c$ and $q^d$;
    \item[] $\cU_q^0$ the subalgebra generated by $\cH_q$ and $\Uqhh$;
    \item[] $\cU^\pm_{q}$ the subalgebra generated by $X^{\pm}_{i,r}$, $i\in 
    I, r\in \Z$;
    \item[] $\cP_{q}$ the subalgebra generated by $\cU_q^0$ and $\cU^+_{q}$;
    \item[] $\B_{q}$ the subalgebra generated by  $\Uqhh$, $\cH^{+}_q$ and $\cU^+_{q}$;
    \item[] $\cN_{q}^{\pm\varphi}$ the subalgebra generated by $\cH^{\pm \varphi}_q$ and $\cU^+_{q}$.
\end{itemize}

\medskip

\subsection{The PBWD basis}

We recall the PBW basis constructed in \cite{Tsy21}. For this we consider the total orders on $\D$ and on $\D^+\times \Z$ as in Section~\ref{sec:irr.criteria.classical}. For each $(\alpha,r)\in \D^+\times \Z$, write $\alpha=\alpha_{i_1}+\cdots +\alpha_{i_p}$, ${\alpha_{i_j}}\in \Sigma$, and fix:
\begin{enumerate}
    \item a decomposition $r=r_1+\cdots r_p$, $r_i \in \Z$;
    \item a sequence $(q_1,\dots,q_{p-1})\in \{q, q^{-1}\}^{p-1}$.
\end{enumerate}
Now, each $(\alpha,r)\in \D^+\times \Z$ defines a vector $X_{\pm \alpha,r}\in \cU_q^\pm$ as
\begin{align*}
    X_{\pm\alpha,r}:=[\cdots[[X^\pm_{i_1,r_1},X^\pm_{i_2,r_2}]_{q_1}, X^\pm_{\alpha_{i_3},r_3}]_{q_3},\dots,X^\pm_{i_p,r_p}]_{q_{p-1}}.
\end{align*}
We have that $K_iX_{\pm\alpha,r}K_i^{-1}=q^{\pm\alpha(h_i)}X_{\pm\alpha,r}$, that is, $X_{\pm\alpha,r}$ is a root vector whose associated root is $\alpha+r\delta\in \D^+\times \Z$.

As in Section~\ref{sec:irr.criteria.classical}, we let $\mathsf{M}$ be the set of all functions $\mathsf{m}:\D^+\times \Z \rightarrow \Z_{\geq 0} $ with finite support and such that $\mathsf{m}(\alpha,k)\leq 1$ if $|\alpha|=1$, and, for each $\mathsf{m}\in \mathsf{M}$ we define monomials 
\begin{align*}
    &X^+_\mathsf{m}:=\prod_{(\alpha,r)\in \D^+\times \Z}^{\rightarrow} (X_{\alpha,r})^{\mathsf{m}(\alpha,r)},  &&X^-_\mathsf{m}:=\prod_{(\alpha,r)\in \D^+\times \Z}^{\leftarrow} (X_{-\alpha,r})^{\mathsf{m}(\alpha,r)}.
\end{align*}
These monomials are called ordered PBWD monomials of $\cU_q^+$ and $\cU_q^-$, respectively. As before, the arrow $\rightarrow$  over the product indicates that the product is written in increasing order from left to right with respect to the order  on $\D^+\times \Z$, and the opposite for $\leftarrow$.

\begin{thm}{\cite[Theorem~5.7]{Tsy21}}\label{thm:PBWD}
The set of ordered monomials $\{X^\pm_\mathsf{m}\mid \mathsf{m}\in \mathsf{M}\}$ forms a linear basis of $\cU_q^\pm$.
\end{thm}

Finally, let $\mathsf{H}$ be the set of all functions $\mathsf{h}:I\times \Z \rightarrow \Z $ with finite support and such that $\mathsf{h}(i,r)\geq 0$ if $r\neq 0$, and define the monomials 
    \[
K^\pm_\mathsf{h}:=\prod_{(i,r)\in I\times \Z_{>0}} (K^\pm_{i,\pm r})^{\mathsf{h}(i,\pm r)},\quad K^0_\mathsf{h}:=\prod_{i\in I} (K_{i})^{\mathsf{h}(i,0)},\quad K_\mathsf{h}:=K^-_\mathsf{h}K^0_\mathsf{h}K^+_\mathsf{h}.
    \]

\begin{thm}{\cite[Remark~5.10]{Tsy21}}\label{thm:PBWD2}
The set of monomials 
    \[
\{X^-_\mathsf{m} K_\mathsf{h} q^{rc} q^{r'd} 
X^+_{\mathsf{m}'}\mid \mathsf{m}, \mathsf{m}'\in \mathsf{M}, \mathsf{h}\in \mathsf{H}, r,r' \in \Z\}
    \]
is a linear basis of $\Uqgh$. In particular, we have a triangular decomposition $\Uqgh\cong \cU_q^-\otimes \cU_q^0\otimes \cU_q^+$.
\end{thm}

\medskip

\subsection{Quantum induced modules}

Set $P_{\Z}:= \{\lambda\in \h^*\mid \lambda(h_i) \in \Z\ \forall i\in I\}$. Let $V_q \in \mathscr{K}_q$ and $\lambda \in P_{\Z}$. Extend $V_q$ to a $\cP_{q}$-module by setting 
    \[
X^+_{i,n}V_q=0, \,\,\,K_i^{\pm}v=q^{\pm \lambda(h_i)}v
    \]
for any $v\in V_{q}$. Define the induced $\Uqgh$-module $M_q(\lambda, V_q)$ as
\begin{align*}
    M_q(\lambda, V_q):=\Ind({\cP_q}, {\Uqgh}; V_q).
\end{align*}
If $V_q$ is an irreducible $\cH_q$-module, then $M_q(\lambda, V_q)$ is called a \emph{quantum generalized imaginary Verma module} associated with $\cP_{q}$, $V_q$ and $\lambda$. If $V_q$ is a $\varphi$-Verma module of $\cH_q$, then $M_q^\varphi(\lambda):=M_q(\lambda, V_q)$ is called a \emph{quantum imaginary $\varphi$-Verma module} of $\Uqgh$.

We have the following basic properties of $M_q(\lambda, V_q)$:

\begin{prop}\label{prop.gen.loop}
Let $V_q$ be an irreducible $\cH_q$-module in $\mathscr{K}_q$ with a $\C(q)$-linear basis $\{v_j\}_{j\in J}$, and let $\lambda \in \hh^*$ with $\lambda(c)\neq 0$. Then,
\begin{enumerate}
    \item \label{gen.loop.1} $M_q(\lambda, V_q)$ is a free $\cU^-_q$-module generated by $\{v_j\}_{j\in J}$;
    \item \label{gen.loop.2} $\dim M_q(\lambda, V_q)_\mu=\infty$ for any $\mu = \lambda-\beta +n\delta$, $\beta \in Q^+\setminus\{0\}$, $n\in \Z$;
    \item \label{gen.loop.3} $\dim M_q(\lambda, V_q)_\mu<\infty$ if and only if $V_q$ is a $\varphi$-Verma module of $\cH_q$ with $\varphi$ constant and $\mu=\lambda- \varphi(n)n\delta$, $n\in \Z_{\geq 0}$;
    \item \label{gen.loop.4} $M_q^\varphi(\lambda)$ is a free $\cN^{-\varphi}_q$-module generated by $1\otimes v_{\varphi}$.
\end{enumerate}
\begin{proof}
The statements \eqref{gen.loop.1} and \eqref{gen.loop.2} follow from the triangular decomposition $\Uqgh= \cU_q^-\otimes  \cU_q^0 \otimes \cU_q^+$ and the PBW basis of $\cU_q^-$. Item \eqref{gen.loop.4} is similar, noting that we also have the decomposition $$\Uqgh= \cN^{-\varphi}_q\otimes \Uqhh \otimes \cN^{\varphi}_q.$$
To prove item \eqref{gen.loop.3}, notice that if $\dim M_q(\lambda, V_q)_\mu<\infty$, then $\mu=\lambda+ n\delta$ for some $n\in \Z$, and  $M_q(\lambda, V_q)_\mu=1\otimes V_{q,n}$. But $\dim V_{q,n}<\infty$ if and only if $V_{q}$ is a $\varphi$-Verma module of $\cH_q$ with $\varphi$ constant.
\end{proof}
\end{prop}

\subsection{\texorpdfstring{$\A$}{A}-forms}

Recall that $\A = \{\frac{f(q)}{g(q)}\in \C(q)\mid g(1)\neq 0\}$. For $i\in I$, $k,l,n \in \Z$ with $n>0$, and $Y \in \Uqh$, we define
\begin{align*}
    \qK{K_{i}}{k,l}{n}= \prod_{r=1}^n \frac{q^{kc}K^+_{i,k+l} - q^{lc}K^-_{i,k+l}}{q^r-q^{-r}},\\
\qK{Y}{k}{n}= \prod_{r=1}^n \frac{Yq^{k-r+1} - Y^{-1}q^{r-k-1}}{q^r-q^{-r}}.
\end{align*}
Let $U_\A (\gh)$ denote the $\A$-subalgebra of $\Uqgh$ generated by $X^\pm_{i,r}, H_{i,\pm n}$, $K^{\pm 1}_i, q^{\pm c}, q^{\pm d}$, $\qK{K_i}{k}{n}$, $\qK{q^c}{k}{n}$, $\qK{q^d}{k}{n}$, $\qK{K_{i}}{k,l}{n}$, $i \in I,\; k, l, n, r\in \Z$, $n>0$.

By straightforward calculations, we have the following proposition.

\begin{prop}\label{prop:gen.UA}
The generators of $U_\A(\gh)$ satisfy the following relations.
\begin{enumerate}[itemsep=3mm]
    \item $q^{\pm c} \text{ is central},\quad K_iK_j=K_jK_i,\quad K_ix^\pm_j(z)K_i^{-1}=q^{\pm A_{i,j}}X^\pm_j(z)$;
    \item $q^d X^\pm_{i,r} q^{-d}= q^r X^\pm_{i,r}  $;
    \item $\qK{K_{j}}{k}{n}X^\pm_{i,r}=X^\pm_{i,r}\qK{K_{j}}{k\pm A^\s_{i,j}}{n}$;
    \item $\qK{q^d}{k}{n}X^\pm_{i,r}=X^\pm_{i,r}\qK{q^d}{k+r}{n}$;
    \item $\qK{q^d}{k}{n}H^\pm_{i,r}=H^\pm_{i,r}\qK{q^d}{k+r}{n}$;
    \item $[H_{i,k},H_{j,l}]=\delta_{k-l,0}\,\dfrac{[kA_{i,j}]}{k}\qK{q^c}{k}{1}$;
    \item $[X^+_{i,r},X^-_{j,s}]=\delta_{i,j}\qK{K_{i}}{r,s}{1}$.
    \item $[H_{i,r},X^{\pm}_j(z)]=\pm\frac{[rA_{i,j}]}{r}q^{-c(r\pm|r|)/2}z^rX^\pm_j(z),$
    
\end{enumerate}
\end{prop}

As a consequence of Proposition~\ref{prop:gen.UA} we have the following corollary. The proof in the supersymmetric case is identical to that in the even case given in \cite{CFM}.

\begin{cor} Let $\cU_\A^\pm = U_\A(\gh)\cap \cU_q^\pm$ and $\cU_\A^0=U_\A(\gh)\cap\cU_q^0$. Then, the subalgebra $U_\A(\gh)$ has the triangular decomposition
\begin{align*}
    U_\A(\gh)=\cU_\A^-\otimes \cU_\A^0\otimes \cU_\A^+.
\end{align*}
\end{cor}

Note that the generators of $\cU_\A^\pm$ and  $\cU_q^\pm$ coincide, and all coefficients obtained through commuting the generators lie in $\A$. Therefore, the $\A$ version of Theorem~\ref{thm:PBWD} holds, i.e., the PBWD monomials $X^\pm_\mathsf{m}, \mathsf{m}\in \mathsf{M}$, form a linear basis of $\cU_\A^\pm$.

\medskip

Now we construct the $\A$-form of diagonal $\cH_q$-modules.

Let $\bs \mu \in \C(q)^{I\times \Zg}$ and 
    \[
V_q\in \{V_q(\bs \mu, a, F)\mid \mathrm{F}\subseteq \mathrm{F}_{\bs \mu,a}\}\quad  (\text{resp. } V\in \{V(\bs \mu, a, F)\mid \mathrm{F}\subseteq \mathrm{F}_{\bs \mu,a}\})
    \]
as defined in Section~\ref{subsec:diagonal_modules} (see also Remark~\ref{rem:V_q.to_V}). Recall that $V_q$ is isomorphic to a quotient of $\cH_q$ by a left ideal $\mathscr{L}_{\bs \mu,a}$ contained in the left ideal generated by the set $\{\phi_{i,k}\phi_{i,-k}-\mu_{i,k},\ q^c-q^a,\ q^d\mid (i,k)\in I\times\Zg\}\cup \{(\phi_{i,-\varphi(i,k)k})^{\frac{\mu_{i,k}}{[ka]}}\mid \frac{\mu_{i,k}}{[ka]}>0,\ \forall (i,k) \in \mathrm{F}_{\bs \mu,a}\}\cup \{(\phi_{i,-\varphi(i,k)k})^{1-\frac{\mu_{i,k}}{[ka]}}\mid \frac{\mu_{i,k}}{[ka]}\leq 0,\ \forall (i,k) \in \mathrm{F}_{\bs \mu,a}\}$. In fact, if $V_q=V_q(\bs \mu, a, F_{\bs\mu,a})$ then $\mathscr{L}_{\bs \mu,a}$ equals this set.

Let $\cH_\A=U_\A(\gh)\cap\cH_q$. Since $\mu_{i,k}\in \C(q)$ for all $(i,k)\in I\times \Zg$, say $\mu_{i,k}=\frac{f_{i,k}(q)}{g_{i,k}(q)}$, we can choose generators of $\mathscr{L}_{\bs \mu,a}$ in $\cH_\A$ (just multiply each $\phi_{i,k}\phi_{i,-k}-\mu_{i,k}$ by $g_{i,k}(q)$). Let $\mathscr{L}_{\bs \mu,a}^\A$ be the ideal of $\cH_\A$ generated by these elements, and define 
\begin{align*}
    V_\A=\cH_\A/\mathscr{L}_{\bs \mu,a}^\A.
\end{align*}

\begin{rem}\label{rem:A-forms}
Usually the $\A$-form of a module $V_q$ is constructed by fixing a suitable basis of $V_q$ and taking the $\A$-form of $V_q$ to be the $\A$-module generated by such basis. Then it is proved that such $\A$-module is invariant under the action of the algebra generators. We notice that this construction coincides with ours in the case that $\bs \mu \in \A^{I\times \Z_{>0}}$, i.e., there is no $\mu_{i,k}$ with pole at $q=1$. Hence, assuming that $\bs \mu \in \A^{I\times \Z_{>0}}$, we have that the basis $\{v_j\}_{j\in J}$ of $V_q$ constructed in  Section~\ref{subsec:diagonal_modules} is also an $\A$-basis for $V_\A$. Moreover, we will see below that if $\mu_{i,k}$ has a pole at $q=1$, then the classical limit of $V_\A$ will vanish.
\end{rem}

The $\A$-form $M_\A(\lambda, V_\A)$ of $M_q(\lambda, V_q)$ is defined as the $U_\A$-submodule of $M_q(\lambda, V_q)$ generated by $V_\A$. That is 
\begin{align*}
    M_\A(\lambda, V_\A):=\Ind({\cP_\A},{U_\A}; V_\A).
\end{align*}
As a consequence of the PBW Theorem and the triangular decompositions of $U_\A$, we have the $\A$-form version of Proposition~\ref{prop.gen.loop}.

\begin{lem}
Let $V_q$ be a $\cH_q$-module as above and assume that there is no $\mu_{i,k}$ with pole at $q=1$. Let $\{v_j\}_{j\in J}$ be an $\A$-basis of $V_\A$ as in Remark~\ref{rem:A-forms}, and let $\lambda \in P$ with $\lambda(c)\neq 0$. Then the following statements hold:
\begin{enumerate}
    \item \label{gen.loop.A1} $M_\A(\lambda, V_\A)$ is a free $\cU^-_\A$-module generated by $\{1\otimes v_j\}_{j\in J}$;
    \item \label{gen.loop.A2} $M_\A^\varphi(\lambda)$ is a free $\cN^{-\varphi}_\A$-module generated by $1\otimes v_{\varphi}$.
\end{enumerate}
\end{lem}

Recall that $M_q(\lambda, V_q)$ has a basis $\{X^-_\mathsf{m} v_j\, \mid\, \mathsf{m}\in \mathsf{M},\, j \in J\}$. Thus, we have a $\C(q)$-linear map 
$$M_q(\lambda, V)\rightarrow \C(q)\otimes_\A M_\A(\lambda, V_\A)$$ given by $X^-_\mathsf{m} v_j \mapsto 1\otimes X^-_\mathsf{m} v_j$, whose inverse is given by $f\otimes v \mapsto fv$ for all $f \in \C(q)$ and $v \in M_\A(\lambda, V_\A)$. Notice that this isomorphism preserves weight spaces. That is, we have the following lemma.

\begin{lem} The following isomorphisms of $\C(q)$-vector spaces hold:
\begin{enumerate}
    \item $M_q(\lambda, V_q)\cong \C(q)\otimes_\A M_\A(\lambda, V_\A)$;
    \item $M_q(\lambda, V_q)_\mu\cong \C(q)\otimes_\A M_\A(\lambda, V_\A)_\mu$, where $M_\A(\lambda, V_\A)_\mu=M_q(\lambda, V_q)_\mu\cap M_\A(\lambda, V_\A)$;
    \item $M_\A(\lambda, V_\A)=\bigoplus_{\mu \in \hh^*} M_\A(\lambda, V_\A)_\mu$.
\end{enumerate}
\end{lem}

\medskip

\subsection{Classical limit}

Consider $U'=\A/\langle q-1\rangle\otimes_\A U_\A(\gh)$ and $\ol{U}=U'/K'$, where $K'$ is the ideal of $U'$ generated by $\{K_i-1,\ q^c-1, q^d - 1\}$. Let $\ol{X}$ denote the image in $\ol{U}$ of an element $X\in U'$. It was proved in \cite[Lemma~6.6.1]{Yam96} that $\ol{U}\cong U(\gh)$. Moreover, this isomorphism maps the generator $\ol{X_{\alpha, k}}$ to $x_{\alpha, k}$ for all $\alpha\in \hD$ and $k\in \Z$.

Let $M_\A$ be equal to $V_\A$ or $M_\A(\lambda, V_\A)$ as considered in the previous section. We define the classical limit of $M_\A$ to be
\begin{align*}
    \overline{M}_\A=\A/\langle q-1\rangle \otimes_\A M_\A.
\end{align*}
For all $m\in M_\A$, we will denote the vector $1\otimes m\in \overline{M}_\A$ by $\overline{m}$. Since $\cH \cong \A/\langle q-1\rangle \otimes_\A \cH_\A$, we have that $\overline{V}_\A$ is a module over $\cH$. Moreover, since $K'\overline{M}_\A=0$ we also conclude that $\ol{M_\A(\lambda, V_\A)}$ is a module over $U(\gh)$.

\begin{rem}\label{rem-faithful}
Assume that, for some $(i,k)\in I\times \Zg$, we have $\mu_{i,k}=\frac{f(q)}{(q-1)g(q)}$, that is, $\mu_{i,k}$ has a pole at $q=1$. In this case $\overline{V}_\A=0$, and hence $\ol{M_\A(\lambda, V_\A)}=0$.
\end{rem}

Recall that 
    \[
V_q\in \{V_q(\bs \mu, a, F)\mid \mathrm{F}\subseteq \mathrm{F}_{\bs \mu,a}\}\quad  (\text{resp. } V\in \{V(\bs \mu, a, F)\mid \mathrm{F}\subseteq \mathrm{F}_{\bs \mu,a}\}).
    \]

Remark \ref{rem-faithful} motivates the following definition.

\begin{dfn}\label{def-limit-faithful}
We call $V_q$ \emph{limit faithful} if there is no $\mu_{i,k}$ with pole at $q=1$.
\end{dfn}

\begin{lem}\label{lem:clasical_limit=classical_module}
Assume that $V_q$ is limit faithful. Then we have isomorphisms: $\ol{V}_\A\cong V$ as $\cH$-modules, and $\ol{M_\A(\lambda, V_\A)}\cong M(\lambda, V)$ as $\gh$-modules.
\end{lem}
\begin{proof}
The former isomorphism follows from the fact that the linear isomorphism that maps $\ol{v_j}$ to $v_j$ is a homomorphism of $\cH$-modules. For the latter isomorphism, we similarly consider the linear isomorphism that maps $\ol{X^-_\mathsf{m} v_j}$ to $x^-_\mathsf{m} v_j$. The fact that this is a homomorphism of $\gh$-modules follows from the fact that $\ol{U}\cong U(\gh)$ and that such isomorphism of algebras maps the generator $\ol{X_{i, k}^\pm}$ to $x_{i, k}^\pm$ for all $i\in I$ and $k\in \Z$. Thus, for any $\ol{X^-_\mathsf{m}}$ of $\ol{U}$ we have that the PBW expansion of $\ol{X_{i,k}^\pm}\ \ol{X^-_\mathsf{m}}$ in $\ol{U}$ is mapped to the PBW expansion of $x_{i,k}^\pm x^-_\mathsf{m}$ in $U(\gh)$.
\end{proof}

\begin{lem}\label{lem.ht.quantum}
Let $V_q$ be limit faithful, and let $\lambda \in P_{\Z}$ with $\lambda(c)\neq 0$. Then, for any nonzero weight vector $v\in M_q(\lambda, V_q)$ with $\htt{v}>0$, there exists $X \in \Uqgh$ such that $Xv\neq 0$ and $\htt{Xv}<\htt{v}$.

\begin{proof}
Let $v\in M_q(\lambda, V)$ be a nonzero vector of weight $\mu$ such that $\htt{v}>0$. Thus, $\mu$ must be of the form $\mu=\lambda-\beta +n\delta$, for some $\beta \in Q^+\setminus\{0\}$, $n\in \Z$. Write 
$$v=\sum_{j\in L}a_j(q) X^-_{\mathsf{m}_j} w_j$$ as in Lemma~\ref{lem.ht2}.

Without loss of generality we may assume that $a_j(q)\in \A$ for all $j\in L$ and $a_j(1)\neq 0$ for some $j\in L$. Hence $v\in M_\A(\lambda, V_\A)\setminus\{0\}$, and $$\ol{v}=\sum_{j\in L}a_j(1) \ol{X^-_{\mathsf{m}_j}} \ol{v_j}\neq 0$$ in $\overline{M_\A(\lambda, V_\A)}\cong M(\lambda, V)$ by Lemma~\ref{lem:clasical_limit=classical_module}. By Lemma~\ref{lem.ht2}, there exists $x_{\alpha, k}\in \cU^+$ such that $x_{\alpha, k}\ol{v}\neq 0$ and $\htt{x_{\alpha, k}\ol{v}}<\htt{\ol{v}}$. This implies $X_{\alpha, k}v\neq 0$. Since $\htt{X_{\alpha, k}v}<\htt{v}$ the statement follows.
\end{proof}
\end{lem}

\begin{thm}\label{thm-irred}
Let $V_q$ be limit faithful, and let $\lambda \in P_{\Z}$ with $\lambda(c)\neq 0$. For any non-trivial submodule $W\subseteq M_q(\lambda, V_q)$, we have $W\cap (1\otimes V_q)\neq 0$. 

\begin{proof}
Let $w\in W$ and assume $w$ is a weight vector. If $\htt{w}=0$, then $w\in 1\otimes V$. If $\htt{w}>0$, then by Lemma~\ref{lem.ht.quantum}, there exists $X\in \cU_q^+$ such that $Xw\neq 0$ and the statement follows by induction.

\end{proof}
\end{thm}

From Theorem \ref{thm-irred} we immediately obtain the following irreducibility criterion for quantum generalized imaginary Verma modules.

\begin{cor}\label{cor-irred}
Let $V_q$ be limit faithful, and let $\lambda \in P_{\Z}$ with $\lambda(c)\neq 0$. Then $M_q(\lambda, V_q)$ is irreducible if and only if $V_q$ is irreducible.
\end{cor}

In particular, we get the irreducibility criterion for quantum imaginary $\varphi$-Verma modules.

\begin{cor}
Let $\lambda \in P_{\Z}$ with $\lambda(c)\neq 0$ and $\varphi:I\times \Z_{>0}\rightarrow \{\pm\}$. Then the quantum imaginary $\varphi$-Verma module $M_q^\varphi(\lambda)$ is irreducible.
\end{cor}

\section*{Acknowledgments}
	L.\,B.\ is supported by the FAPESP (2020/14281-5). L.\,C.\ is supported by the FAPEMIG (APQ02768-21). V.\,F.\ is supported in part by the CNPq (304467/2017-0) and by the Fapesp (2018/23690-6).


\medskip

\begin{thebibliography}{}

\bibitem[BBFK]{BBFK} V.~Bekkert, G.~Benkart, V.~Futorny, and I.~Kashuba. {\it New irreducible modules for Heisenberg and affine Lie algebras}, Journal of Algebra {\bf 373} (2013), 284--298

\bibitem[CF]{CF}
L.~Calixto, V.~Futorny, {\it Highest weight modules for affine Lie superalgebras}, Rev. Mat. Iberoam. {\bf 37} (2021), no. 1, pp. 129–160

\bibitem[CFM]{CFM} B.~Cox, V.~Futorny and K.~Misra, {\it An imaginary PBW basis for quantum affine algebras of type 1}, Journal of Pure and Applied Algebra {\bf 219} Issue 1  (2015), 83--100

\bibitem[DCK]{DCK} C.~De Concini, V. G. Kac, {\it Representations of quantum groups at roots of 1}, Operator algebras, unitary representations, enveloping algebras, and invariant theory (Paris, 1989), 471--506, Progr. Math., {\bf 92}, Birkhäuser Boston, Boston, MA, 1990.

\bibitem[EF]{EF}
Eswara~Rao, V. Futorny, V.
{\it Integrable modules for affine {L}ie superalgebras},
Trans. Amer. Math. Soc. \textbf{361} (2009), no. 10, 5435--5455.

\bibitem[FGM14]{FGM14} V.~Futorny, D.~Grantcharov, and V.~Mazorchuk, {\it Weight modules over infinite dimensional Weyl algebras}, Proc. Amer. Math. Soc. {\bf 142} (2014), 3049--3057

\bibitem[FGM]{FGM}  V. Futorny, A. Grishkov, D. Melville, {\it Verma type modules for Quantum Affine Lie algebras}, Algebras and Representation Theory 8 (2005), 99--125.

\bibitem[FHW]{FHW} Futorny, V., Hatrwig J., Wilson, E. {\it Irreducible completely pointed modules of quantum groups of type A},  J. Algebra,  {\bf 432} (2015), 252--279.



\bibitem[FK1]{FK1} V. Futorny and I. Kashuba,  {\it Induced modules for Kac-Moody Lie algebras},  SIGMA - Symmetry, Integrability and Geometry: Methods and Applications {\bf 5} (2009), Paper 026.


\bibitem[FK2]{FK2} V. Futorny and I. Kashuba,  {\it 
Structure of parabolically induced  modules for affine Kac--Moody algebras}, J. Algebra \textbf{500} (2018), 362--374.


\bibitem[FKS]{FKS} 
V. Futorny, L. K\v{r}i\v{z}ka and P. Somberg, {\it Geometric  realizations of affine Kac--Moody algebras}, J. Algebra \textbf{528} (2019),  177--216.

\bibitem[GKMOS]{GKMOS}  M. Guerrini, I. Kashuba, O. Morales, A. de Oliveira, F. Junior Santos, {\it Generalized Imaginary 
Verma and Wakimoto modules}, \url{https://arxiv.org/abs/2201.07617}

\bibitem[Lu]{Lu} G. Lusztig, {\it Quantum deformations of certain simple modules over enveloping algebras}, Adv. Math. 70 (1988), 237–249.


\bibitem[Tsy21]{Tsy21} A.~Tsymbaliuk,
		{\it PBWD bases and shuffle algebra realizations for 
			$U_{\boldsymbol{v}}(L\mathfrak{sl}_n)$, $U_{\boldsymbol{v_1},\boldsymbol{v_2}}(L\mathfrak{sl}_n)$, $U_{\boldsymbol{v}}(L\mathfrak{sl}_{m|n})$}, Sel. Math. New Ser. {\bf 27}, 35 (2021). \url{https://doi.org/10.1007/s00029-021-00634-5}
	

\bibitem[Yam96]{Yam96} H. Yamane,
	{\it On defining relations of affine Lie superalgebras and affine quantized universal enveloping superalgebras}, 
	Publ. RIMS, Kyoto Univ. {\bf 35} (1999), 321-–390
\end{thebibliography}
\end{document}